\documentclass[11pt, reqno]{amsart}   	
		
\usepackage{accents}	
\usepackage{adjustbox}
\usepackage{amssymb, latexsym}
\usepackage{amsmath}
\usepackage[toc,page]{appendix}
\usepackage{braket}
\usepackage{breqn}
\usepackage{caption}
\usepackage{color}
\usepackage{comment}
\usepackage{dsfont}
\usepackage{esint}
\usepackage{enumitem}
\usepackage[T1]{fontenc}
\usepackage{geometry}         
\usepackage{graphicx}
\usepackage{lipsum}
\usepackage{hyperref}
\usepackage{latexsym}
\usepackage{mathrsfs}
\usepackage{mathtools}
\usepackage{subcaption}
\usepackage{float}
\restylefloat{table}
\usepackage{listings}
\usepackage[dvipsnames]{xcolor} 

\usepackage{times}
\usepackage{url} 
\usepackage{tikz}
\usetikzlibrary{cd}
\usetikzlibrary{positioning}

\theoremstyle{plain}
\numberwithin{equation}{section}
\newtheorem{thm}{Theorem}[section]
\newtheorem{theorem}[thm]{Theorem}
\newtheorem{lemma}[thm]{Lemma}

\newtheorem{assumption}[thm]{Assumption}

\newtheorem{corollary}[thm]{Corollary}

\newtheorem{proposition}[thm]{Proposition}

\newtheorem{remark}[thm]{Remark}
\newtheorem{problem}[thm]{Problem}

\def\ep{\epsilon}
\def\ga{\gamma}

\def\la{\lambda}

\def\om{\omega}
\def\Om{\Omega}
\def\pa{\partial}

\def\N{\mathbb{N}}

\def\R{\mathbb{R}}

\def\fa{\forall}
\def\grad{\nabla}

\def\lap{\bigtriangleup}

\def\Alg{\mathcal{A}}

\def\Lin{\mathcal{L}}

\def\des{\text{des}}

\def\liminf{\text{liminf}}
\def\limsup{\text{limsup}}

\def\loc{\text{loc}}

\newcommand\beq{\begin{equation}}
\newcommand{\bburl}[1]{\textcolor{blue}{\url{#1}}}
\newcommand\eeq{\end{equation}}
\newcommand\bea{\begin{eqnarray}}
\newcommand\eea{\end{xq}}
\newcommand\bi{\begin{itemize}}
\newcommand\ei{\end{itemize}}
\newcommand\ben{\begin{enumerate}}
\newcommand\een{\end{enumerate}}

\everymath{\displaystyle}

\begin{document}
\begin{center}
\uppercase{\bf \boldmath Existence of Solutions for Fractional Optimal Control Problems with Superlinear-subcritical Controls}
\vskip 20pt
{\bf Joshua M. Siktar}\\
{\textit{Department of Mathematics, Texas A\&M University, TX, USA}}\\
{\tt Email: jmsiktar@tamu.edu,  Phone: 412-860-3371}\\ 


\date{\today}							
\end{center}

\vskip 30pt
\centerline{\bf Abstract}
\noindent
    This paper gives an existence result for solutions to an elliptic optimal control problem based on a general fractional kernel, where the admissible controls come from a class satisfying both a growth bound and a superlinear-subcritical condition. Each admissible control is known to produce a nontrivial corresponding state by applying the Mountain Pass Theorem to fractional equations. The main theoretical contribution is the construction of a suitable set of admissible controls on which the standard existence theory for control problems with linear and semi-linear state constraints can be adapted. Extra care is taken to explain what new difficulties arise for these types of control problems, to justify the limitations of this theory. For completeness, the corresponding local elliptic control problem is also studied.



\section{Introduction}\label{sec: intro}

The goal of this paper is to prove the existence of solutions to a fractional optimal control problem, where each control has multiple associated states governed by solving a semi-linear fractional equation whose associated energy is studied using minimax methods. The state equation in our setting is of the form
\begin{equation}\label{Eq: MacroFracEqn}
    -\Lin_K u(x) \ = \ g(x, u(x)),
\end{equation}
which is solved pointwise for $x$ belonging to a bounded domain $\Om \subset \R^n$, where $u$ is prescribed to vanish on $\R^n\setminus \Om$. Here $\Lin_K$ is an elliptic integro-differential operator; a special case is where $\Lin_K := -(\lap)^s$, where $-(\lap)^s$ denotes the Fractional Laplacian operator given in the Cauchy principal value sense; this operator has been thoroughly studied in references to be given later. 

The mathematical novelty of this work is in using an equation of the form \eqref{Eq: MacroFracEqn} as a constraint for an optimal control problem, which takes the general form
\begin{equation}\label{Eq: MacroOptControl}
    \begin{cases}
        \min\{I(u, g) \ | \ (u, g) \in X_{\text{ad}} \times Z_{\text{ad}}\} \\
        -\Lin_K u(x) \ = \ g(x, u(x)), \ x \in \Om
    \end{cases},
\end{equation}
where the exact form of the admissible sets $X_{\text{ad}}$ and $Z_{\text{ad}}$, as well as the cost functional $I(\cdot, \cdot)$, will be specified in due time. This framework can serve as a selection criteria among a collection of possible physical states associated with a given input control. 

Now we highlight the literature that gives context for the problem to be addressed in the current work. The solving of an [elliptic] partial differential equation (PDE) is equivalent to finding critical points of some energy minimizer. If the energy to be minimized is strictly convex, that corresponds to the equation to be solved having a unique solution. In order to understand the non-convex case, two classes of theories have been developed: Morse theory (classical texts include \cite{milnor1963morse, schwartz1969nonlinear}), and minimax theory (classical texts include \cite{ljusternik1934methodes, birkhoff1917dynamical}, and see \cite{rabinowitz1986minimax} for a more recent survey; Chapter 2 of \cite{tavares2024topics} presents a comprehensive overview of different methods for treating semi-linear elliptic PDE). The types of state equations studied in this paper have their properties proven as applications of minimax theory, in particular with a central result known as the \textbf{Mountain Pass Theorem}. This theorem acts as a sufficient condition for a non-convex energy to have multiple critical points; in this setting, one corresponds to the trivial solution of the PDE, so there remains at least one non-trivial solution. This theory was originally developed to study the multiplicity of second-order semi-linear PDEs where the main term of the differential operator was canonically a Laplacian, including applications to Hamiltonian systems and bifurcation theory. Resulting works in this direction include \cite{ambrosetti1973dual, rabinowitz1986minimax, struwe2000variational, de2009multiple, choi1993mountain, choi1993semiwave, degiovanni2022variational, clark1972variant, kajikiya2005critical}; one variant of this result will be stated precisely later in the paper. Other, more recent papers, present further sets of conditions under which an energy has two, and sometimes more, distinct critical points, in an abstract functional analytic framework; these include \cite{ricceri2000existence, ricceri2000three, ricceri2010multiplicity, ricceri2011further, ambrosetti1973dual, ambrosetti1992critical}

More recent papers have studied minimax methods for nonlocal equations, particularly those of fractional type using the same tools as used to study PDEs. Nonlocal models, including fractional ones, have emerged as a popular way to study long-range interactions, utilizing an integration around a fixed point in space. Reference texts on the analytic properties of Fractional Sobolev Spaces include \cite{demengel2012functional, leoni2023first, di2012hitchhikers}, and fractional equations were studied using minimax theory in \cite{servadei2012mountain, servadei2013variational, bisci2014mountain, bisci2014three, gu2018infinitely, bors2014application, chammem2023combined, dipierro2017fractional, bisci2016variational, maione2023variational}, among other sources. Finally, we comment that a few recent papers have studied optimal control problems using nonlocal models in the constraints, see \cite{d2019priori, mengesha2023control, han2024compactness,  munoz2022local}.

We now outline the contents of the remainder of the paper. Section \ref{sec: prelim} provides notation and formal problem statements, then provides necessary background material. Then, Section \ref{sec: Foundation} discusses the two new core ideas for our optimal control framework: the introduction of the admissible set of controls, and how to ensure that the limit of a sequence of admissible pairs of controls and states still solves the constraint equation. Finally, Section \ref{Sec: Existence} gives our main existence result for the nonlocal control problem, and then the analogous result for the corresponding, PDE-based local problem. 


\section{Preliminaries}\label{sec: prelim}

\subsection{Notation}\label{sec: note}

Throughout this paper we will assume that $\Om \subset \R^n$ is a bounded, open set with a Lipschitz domain, and we will assume that $n > 2$. We also denote $Q := (\Om \times \R^n) \cup (\R^n \times \Om)$. The nonlocal control problem is posed in terms of an abstract kernel $K$ with fractional parameter $s \in (0, 1)$ fixed. These are assumptions on $K$ that we use, and they are the ones that appear in \cite{servadei2012mountain, servadei2013variational}.
\begin{assumption}[Kernel assumptions]\label{Assump: ker}
    The kernel $K: \R^n\setminus\{0\} \rightarrow [0, \infty)$ satisfies the following properties:
    \begin{enumerate}[label=\textbf{K\arabic*}]
        \item\label{K1} If $\ga(x) := \min\{|x|^2, 1\}$, then $\ga K \in L^1(\R^n)$;
        \item\label{K2} There exists $\la > 0$ such that $K(x) \geq \la|x|^{-(n + 2s)}$ for all $x \in \R^n\setminus\{0\}$;
        \item\label{K3} $K$ is even; i.e. $K(x) = K(-x)$ for all $x \in \R^n\setminus \{0\}$
    \end{enumerate}
\end{assumption}

The prototypical kernel that satisfies Assumption \ref{Assump: ker} is the fractional kernel of order $s$, i.e. $K(x) = |x|^{-(n + 2s)}$. In this and the general case, we define our nonlocal operator $\Lin_K$ as
\begin{equation}\label{Eq: NLOperatorLK}
    \Lin_K u(x) \ := \ \frac{1}{2}\int_{\R^n}(u(x + y) + u(x - y) - 2u(x))K(y)dy
\end{equation}
for each $x \in \R^n$. Meanwhile, for the local problem we will simply use the Laplacian operator $-\lap$.

Now we may define our function spaces. The classical Fractional Sobolev Space is defined as, for $s \in (0, 1)$,
\begin{equation}\label{Eq: HsFuncSpace}
    H^s(\Om) \ := \ \{u: \R^n \rightarrow \R \ | \ u|_{\Om} \in L^2(\Om), \ \|u\|_{H^s(\Om)} \ < \ \infty\},
\end{equation}
where the norm $\|\cdot\|_{H^s(\Om)}$ is defined by
\begin{equation}\label{Eq: HSNorm}
    \|u\|_{H^s(\Om)} \ := \ \|u\|_{L^2(\Om)} + \left(\iint_{\Om \times \Om}\frac{|u(x) - u(y)|^2}{|x - y|^{n + 2s}}dxdy\right)^{\frac{1}{2}}.
\end{equation}
In the case where we want to pose a Dirichlet boundary condition, we denote that space as
\begin{equation}\label{Eq: Hs0}
\widetilde{H^s}(\R^n) \ := \ \{u \in H^s(\Om) \ | \ u|_{\R^n\setminus \Om} \ \equiv \ 0\}.
\end{equation}
It is well-known that $(H^s(\Om), \|\cdot\|_{H^s(\Om)})$ and $(\widetilde{H^s}(\R^n), \|\cdot\|_{H^s(\Om)})$ are Hilbert spaces for any $s \in (0, 1)$.

We also need nonlocal function spaces associated with the kernel $K$, which we denote as
\begin{equation}\label{Eq: XFunc}
    X(\Om) \ := \ \Big\{u: \R^n \rightarrow \R \ | \ u|_{\Om} \in L^2(\Om), \ \iint_Q K(x - y)|u(x) - u(y)|^2dxdy \ < \ \infty\Big\},
\end{equation}
where $Q := \R^{2n}\setminus((\R^n\setminus \Om) \times (\R^n\setminus \Om))$.

Solutions for our equations will be posed on the subspace of $X(\Om)$ that possesses a boundary constraint, namely the space
\begin{equation}\label{Eq: X0Func}
    \widetilde{X}(\R^n) \ := \ \{u \in X(\Om) \ |  \ u|_{\R^n\setminus \Om} \ \equiv \ 0\}.
\end{equation}
Naturally, the underlying norm is
\begin{equation}\label{Eq: XNorm}
    \|u\|_{X(\Om)} \ := \ \|u\|_{L^2(\Om)} + [u]_{X(\Om)},
\end{equation}
where $[\cdot]_{X(\Om)}$ is a semi-norm on $X(\Om)$ defined by
\begin{equation}
     [u]_{X(\Om)} \ := \ \left(\iint_Q K(x - y)|u(x) - u(y)|^2dxdy\right)^{\frac{1}{2}}
\end{equation}
We recall that $(\widetilde{X}(\R^n), \|\cdot\|_{X(\Om)}$ is a Hilbert space, see \cite[Lemma 7]{servadei2012mountain}. The definition \eqref{Eq: NLOperatorLK} was also featured there.

For the local problem, we recall the classical Sobolev Space
\begin{equation}\label{Eq: H1}
    H^1(\Om) \ := \ \{u: \R^n \rightarrow \R \ | \ u|_{\Om} \in L^2(\Om), \ \grad u \in L^2(\Om; \R^n)\}
\end{equation}
and define $\widetilde{H^1}(\R^n)$ as the extension $H^1(\Om)$ that vanishes outside of $\Om$, i.e., 
\begin{equation}\label{Eq: H10}
    \widetilde{H^1}(\R^n) \ := \ \{u \in H^1(\Om) \ | \ u \ = \ 0 \ \text{on} \ \R^n\setminus \Om\}
\end{equation}
The natural norm on $H^1(\Om)$ is
\begin{equation}\label{Eq: H1Norm}
    \|u\|_{H^1(\Om)} \ := \ \|u\|_{L^2(\Om)} + [u]_{H^1(\Om)},
\end{equation}
where we have defined
\begin{equation}\label{Eq: H1seminorm}
    [u]_{H^1(\Om)} \ = \ \left(\int_{\Om}|\grad u(x)|^2dx\right)^{\frac{1}{2}}
\end{equation}
to be the canonical semi-norm on $H^1(\Om)$.

Now we may define the admissible sets of controls for the nonlocal and local control problems.

\begin{assumption}[Nonlocal admissible set of controls]\label{assump: AssumptionsOnControls}
 Let $a_1, a_2 > 0$, $\mu > 2$, $r \in (n + 1, \infty)$ be fixed constants, and let $q \in (2, 2^*_s)$ be a fixed exponent where $2^*_s := \frac{2n}{n - 2s}$. Also let $C: (0, \infty) \rightarrow (0, \infty)$ be fixed. We define the admissible set of controls $Z_{\text{ad}}$ to be the collection of functions $g: \Om \times \R \rightarrow \R$ satisfying the following conditions:
    \begin{enumerate}[label=\textbf{A\arabic*}]
    \item \label{A1} We have the inequality
     $\|g\|_{W^{1, r}(\Om \times (-M, M))} \ \leq \ C(M)$ for all $M > 0$. 
        \item \label{A2} We have that \begin{equation}\label{Eq: gqCoercive}
            |g(x, t)| \ \leq \ a_1 + a_2|t|^{q - 1}
        \end{equation}
        for all $x \in \Om$ and $t \in \R$.
        \item\label{A3} We have the limit
        \begin{equation}\label{eq: uniformLimit}
            \lim_{t \rightarrow 0}\frac{g(x, t)}{t} \ = \ 0
        \end{equation}
        uniformly over $x \in \Om$.
        \item\label{A4} If we define 
        \begin{equation}\label{eq: Gantideriv}
            G(x, t) \ := \ \int^{t}_{0}g(x, \tau)d\tau
        \end{equation}
        then there exists $\om > 0$ so for all $t \in \R$ with $|t| \geq \om$, $x \in \Om$ we have
        \begin{equation}\label{Eq: GGG}
            0 \ < \ \mu G(x, t) \ \leq \ tg(x, t)
        \end{equation}
    \end{enumerate}
\end{assumption}

We note that Assumption \ref{A4} is known in the literature as a super-quadraticity condition, or as an Ambrosetti-Rabinowitz condition. Meanwhile, we say the controls have sub-critical growth due to Assumption \ref{A2}.

\begin{assumption}[Local admissible set of controls]\label{Assump: lAdmis}
     Let $a_0, a_1 > 0$, $\mu > 2$ be fixed constants, and let $q \in (2, 2^*_1)$ be a fixed exponent where $2^*_1 := \frac{2n}{n - 2}$
    We assume that a control $g: \Om \times \R \rightarrow \R$ belongs to the admissible set $Z^{\loc}_{\text{ad}}$ if and only if \ref{A1}, \ref{A2}, \ref{A3}, and \ref{A4} are all met, but where $q < 2^*_1$ instead of $q < 2^*_s$.
\end{assumption}

We note that according to the discussion preceding \cite[Theorem 2.15]{rabinowitz1986minimax}, we can also guarantee existence of nontrivial solutions in the local case when $n = 1$ or $n = 2$ with slightly more relaxed growth assumptions. To avoid making the exposition here overly cluttered, however, we opt to present results only in the case $n > 2$.

Now we introduce the equations which will serve as constraints for the nonlocal and local control problems. For the nonlocal problem, we wish to solve the equation
\begin{equation}\label{eq: NLWkStateEqn}
    \begin{cases}
    -\Lin_K u(x) \ = \ g(x,  u(x)), \ x \in \Om \\
 \ \ \ \ \ \ \ \ \ u(x) \ = \ 0, \ \ \ \ \ \ \ \ \ \ \ \ \ \ \ \ x \in \R^n \setminus \Om
    \end{cases}.
\end{equation}
Equivalently, we say that $u$ is a critical point for the energy
\begin{equation}\label{eq: NLEnergy}
    J^K_g(u) \ := \ \frac{1}{2}[u]^2_{X(\Om)} - \int_{\Om}g(x, u(x))u(x)dx
\end{equation}
over all functions $v \in \widetilde{X}(\R^n)$, or that it solves the following equation in the weak sense for all $v \in \widetilde{X}(\R^n)$:
\begin{equation}\label{eq: NLWk}
    \iint_{\R^{2n}}K(x - y)(u(x) - u(y))(v(x) - v(y))dxdy \ = \ \int_{\Om}g(x, u(x))v(x)dx.
\end{equation}
It is easy to see that the first [Fréchet] derivative of $J^K_g$, denoted $(J^K_g)'$, is equal to
\begin{equation}\label{eq: First Frechet}
    (J^K_g)'(u)v \ = \  \iint_{\R^{2n}}K(x - y)(u(x) - u(y))(v(x) - v(y))dxdy - \int_{\Om}g(x, u(x))v(x)dx
\end{equation}
for any $u, v \in \widetilde{H}^s(\R^n)$. Meanwhile, for the local control problem, we wish to solve the equation
\begin{equation}\label{eq: LWkStateEqn}
\begin{cases}
    -\lap u(x) \ = \ g(x,  u(x)), \ x \in \Om \\
 \ \ \ \ \ \ \ \ u(x) \ = \ 0, \ \ \ \ \ \ \ \ \ \ \ \ \ \ \ \ x \in \pa \Om
    \end{cases}.
\end{equation}
Equivalently, we say that $u$ is a critical point for the energy
\begin{equation}\label{eq: LEnergy}
    J^{\loc}_g(u) \ := \ \frac{1}{2}[u]^2_{H^1(\Om)}dx - \int_{\Om}g(x, u(x))u(x)dx,
\end{equation}
or that it solves the following equation in the weak sense for all $v \in \widetilde{H^1}(\R^n)$:
\begin{equation}\label{eq: LWk}
    \int_{\Om}\grad u(x) \cdot \grad v(x) \ = \ \int_{\Om}g(x, u(x))v(x)dx.
\end{equation}

We will also be taking Sobolev Spaces on bounded subsets of $\R^{n + 1}$. If $V$ is such a set and $r \geq 1$, we define the Sobolev Space $W^{1, r}(V)$ as
\begin{equation}\label{eq: SobSpacew}
    \|v\|_{W^{1, r}(V)} \ := \ \{v \in L^r(V) \ | \ \grad v \in L^r(\Om; \R^n)\}
\end{equation}

Another thing we will need for our formal problem statements is the definitions of the cost functionals in the nonlocal and local settings. We choose a quadratic functional of the form
\begin{equation}\label{eq: CostFunc}
    I_s(u, g) \ := \ [u - u_{\text{des}}]^2_{X(\Om)} + \la \int_{\Om}|g(x, S(x))|dx,
\end{equation}
where $\la \geq 0$ is a penalization parameter, and $S \in L^{q - 1}(\Om)$ is fixed, where $q$ is the exponent introduced in Assumptions \ref{assump: AssumptionsOnControls}, \ref{Assump: lAdmis}. 

Meanwhile, in the local setting, we take
\begin{equation}\label{eq: costFuncLoc}
    I_{\loc}(u, g) \ := \ [u - u_{\text{des}}]^2_{H^1(\Om)} + \la \int_{\Om}|g(x, S(x))|dx.
\end{equation}
The physical motivation is that we may wish to displace a material to fit into a hole of a fixed shape represented by the function $u_{\des}$ belonging to $X(\Om)$ in the nonlocal setting, or to $H^1(\Om)$ in the local setting, and this acts as our specific selection criterion. 


\subsection{Problem Statements}\label{sec: probStatement}

Now that we have defined our nonlocal and local state equations, as well as the cost functional, we may define the optimal control problems, starting with the nonlocal one.

\begin{problem}[Nonlocal optimal control problem]\label{Pr: NlCtrlProb}
    Our nonlocal control problem is to find
\begin{equation}\label{stateMinNL}
    I(\overline{u_s}, \overline{g_s}) = \min_{(u_s, g_s) \in \Alg}I(u_s, g_s),
\end{equation} 
where $\Alg$ is defined by
\begin{equation}\label{nonlocalAdmiClass}
    \Alg \ := \{(u, g) \in \widetilde{X}(\R^n) \times Z_{\text{ad}} \ | \ (J^K_g)'(u)v = 0 \ \quad \fa v \in \widetilde{X}(\R^n)\},
\end{equation}
and $Z_{\text{ad}}$ is the admissible set of controls described in Assumption \ref{assump: AssumptionsOnControls}.
\end{problem}

Now we formally state the analogous local problem.

\begin{problem}[Local optimal control problem]\label{Pr: lCtrlProb}
    Our local control problem is to find
    \begin{equation}\label{IMinL}
        I_{\loc}(\overline{u}, \overline{g}) \ = \ \min_{(u, g) \in \Alg^{\loc}}I_{\loc}(u, g)
    \end{equation}
    where $\Alg^{\loc}$ is defined by
    \begin{equation}\label{Eq: AlgLoc}
         \Alg^{\loc} \ := \ \{(u, g) \in \widetilde{H^1}(\R^n) \times Z^{\loc}_{\text{ad}} \ | \ (J^{\loc}_g)'(u)(v) = 0 \ \quad \fa v \in \widetilde{H^1}(\R^n)\},
    \end{equation}
    and $Z^{\loc}_{\text{ad}}$ is the admissible set of controls described in Assumption \ref{Assump: lAdmis}.
\end{problem}


\subsection{Preliminaries}\label{sec: funcAnal}

We begin this subsection by recalling some previously obtained embedding results on the aforementioned function spaces. 

\begin{lemma}\label{lem: EquivNorms}
  Let  $\Om \subset \R^n$ be bounded, open, and Lipschitz. Then the following embedding results hold:
    \begin{enumerate}[label=\textbf{P\arabic*}]
        \item\label{P2} If $n > 2s$, then $H^s(\Om)$ is compactly embedded into $L^q(\Om)$ for every $q \in (1, 2^*_s)$. 
        \item\label{P4} Suppose $r > n + 1$, and $M > 0$ are fixed. Then $W^{1, r}(\Om \times (-M, M))$ is compactly embedded into $C(\Om \times (-M, M))$, the space of continuous, real-valued functions on $\Om \times (-M, M)$ (endowed with the uniform topology).
        \item\label{P5} The space $X(\Om)$ is continuously embedded into $H^s(\Om)$.
    \end{enumerate}
\end{lemma}

Note that \ref{P2} is part of \cite[Theorem 4.54]{demengel2012functional}. In addition, since $\Om \times (-M, M)$ is a bounded subset of $\R^{n + 1}$, \ref{P4} follows from  \cite[Theorem 9.16]{Bre}. Finally, \ref{P5} is part of \cite[Lemma 5]{servadei2012mountain}.

Now we recap the relevant existence results for nontrivial solutions to the local and nonlocal state equations. The two results are quite similar in nature, as they are both applications of the famous Mountain Pass Theorem; the key difference is the critical exponent in the growth condition on the nonlinearity. We state first the result for the nonlocal problem.

\begin{proposition}[Existence of nontrivial solutions to nonlocal state equation]\label{Prop: exsNLState}
    Suppose that $g \in Z_{\text{ad}}$ satisfies Assumption \ref{assump: AssumptionsOnControls}. Then there exists a nontrivial solution $u \in \widetilde{X}(\R^n)$ to Equation \eqref{eq: NLWkStateEqn}.
\end{proposition}

\begin{proof}
    The assumption \ref{A1} is included to ensure the admissible set of controls is weakly closed, and also implies that $g$ is Carathéodory; the other assumptions are those needed to invoke \cite[Theorem 1]{servadei2012mountain}.
\end{proof}

Moreover, note that according to \cite[Corollary 13]{servadei2012mountain}, each admissible control will have at least one non-trivial non-negative solution, and at least one non-trivial non-positive solution. As a result the admissible sets $\Alg$ and $\Alg^{\loc}$ will include pairs where the state is non-positive.

Now we state the analogous result for the local problem \cite{rabinowitz1986minimax} (see also \cite[Section 8.4]{Ev} for a proof written in the case where the right-hand side is autonomous, i.e., $g(x, u(x)) = g(u(x))$).

\begin{proposition}[Existence of nontrivial solutions to local state equation]\label{Prop: exsLState}
    Suppose that $g \in Z^{\loc}_{\text{ad}}$ satisfies Assumption \ref{Assump: lAdmis}. Then there exists a nontrivial solution $u \in \widetilde{H^1}(\R^n)$ to Equation \eqref{eq: LWkStateEqn}.
\end{proposition}

\begin{proof}
    Due to \ref{P4}, the admissible control $g$ automatically satisfies all of the conditions needed to invoke \cite[Theorem 2.15]{rabinowitz1986minimax}.
\end{proof}


\section{Foundations}\label{sec: Foundation}

\subsection{Construction of Admissible Set of Controls}\label{subsec: Zad}

We make some comments on the choice of Assumptions \ref{assump: AssumptionsOnControls} and \ref{Assump: lAdmis}.

The first thing to note is that the admissible set $Z_{\text{ad}}$ is nonempty. The prototypical example of a datum $g$ considered in \cite{servadei2013variational, servadei2012mountain} is the function $g(x, t) \ := \ a(x)|t|^{q - 2}t$ for a function $a \in L^{\infty}(\Om)$. If $a \in W^{1, r}(\Om)$ then it is easy to see that all parts of Assumption \ref{assump: AssumptionsOnControls} are satisfied. Moreover, Assumption \ref{A1} can be viewed as a ``soft" boundedness condition; as we will see, combining this with the embedding \ref{P4} will be sufficient for proving existence of solutions to Problem \ref{Pr: NlCtrlProb}, we do not want to assume admissible controls belong to $W^{1, r}(\Om \times \R)$ as then controls which are super-quadratic in the second argument are no longer admissible.

Now we discuss the need for admissible controls to belong to first-order Sobolev Spaces. Just as in the study of optimal control problems with linear constraints, the key to proving existence of solutions is to have a nonempty, closed, bounded, convex subset of a reflexive Banach space as the admissible set of controls. The important difference in this setting is we need the admissible set of controls to be closed with respect to the uniform topology, instead of merely in a weak $L^p$ topology. This classical theory is described in much detail in the textbook \cite{troltzsch2010optimal}, and some recent papers on nonlocal optimal control problems where the admissible set of controls is a bounded subset of an $L^p$ space include \cite{munoz2022local, mengesha2023control, d2019priori}.

The content of the next lemma is that the admissible set of controls for the nonlocal problem is also convex and closed.

\begin{lemma}[Convex and closed]\label{lem: closedAndConvex}
     If $Z_{\text{ad}}$ satisfies Assumption \ref{assump: AssumptionsOnControls}, then this set is a convex, closed subset of $W^{1, r}(\Om \times (-M, M))$ for any $M > 0$.
\end{lemma}

\begin{proof}
The convexity of $Z_{\text{ad}}$ is immediate. Assumption \ref{A1} is satisfied due to the Triangle Inequality on the $\|\cdot\|_{W^{1, r}(\Om \times (-M, M))}$ norm, whereas \ref{A2}, \ref{A3}, and \ref{A4} are all satisfied by taking convex combinations of the constituent inequalities and equations.

Let $M > 0$ be fixed, and suppose we have a sequence $\{g_k\}^{\infty}_{k = 1} \subset Z_{\text{ad}}$ such that $g_k \rightharpoonup g$ in the weak $W^{1, r}(\Om \times (-M, M))$ topology. Then we show that $Z_{\text{ad}}$ is closed, i.e. that $g \in Z_{\text{ad}}$. First, notice that since $W^{1, r}(\Om \times (-M, M))$ is reflexive, $g$ automatically satisfies \ref{A1}. 

Now, since $r \in (n + 1, \infty)$ and $\Om \times (-M, M)$ is a bounded subset of $\R^{n + 1}$, we have by \cite[Theorem 9.16]{Bre} that $g_k \rightarrow g$ uniformly in $\Om \times (-M, M)$, and so assumptions \ref{A2} and \ref{A4} are immediately satisfied for $g$. Finally, we prove \ref{A3}: since each $g_k$ satisfies \eqref{eq: uniformLimit} we have that
\begin{equation}\label{eq: closedAndConvexUnifConvEq1}
    \lim_{t \rightarrow 0}\frac{g_k(x, t)}{t} \ = \ 0
\end{equation}
uniformly over $x \in \Om$. Since $g_k \rightarrow g$ uniformly on $\Om \times (-M, M)$ we may invoke the Moore-Osgood Theorem to see that
\begin{equation}\label{eq: closedAndConvexUnifConvEq2}
    \lim_{t \rightarrow 0^+}\frac{g(x, t)}{t} \ = \ 0
\end{equation}
uniformly over $x \in \Om$. Thus $g \in Z_{\text{ad}}$.
\end{proof}

Now we note that just as in the nonlocal case, the function $g(x, t) := a(x)|t|^{q - 2}t$, where $a \in W^{1, r}(\Om)$, ensures that $Z^{\loc}_{\text{ad}}$ is nonempty. In addition, we have the following result for the local problem, by the same argument as used to Prove Lemma \ref{lem: closedAndConvex}.

\begin{lemma}[Convex and closed]\label{lem: closedAndConvexL}
     If $Z^{\loc}_{\text{ad}}$ satisfies Assumption \ref{Assump: lAdmis}, then this set is a convex, closed subset of $W^{1, r}(\Om \times (-M, M))$ for any $M > 0$.
\end{lemma}


\subsection{Preserving Admissibility in the Limit}\label{sec: admissibility}

The goal of this subsection is to show that the limit of a sequence of admissible pairs (in the sense of solving the nonlocal state equation) remains admissible. This theory will be applied to minimizing sequences of controls when we prove our main existence result in Section \ref{Sec: Existence}. 

\begin{lemma}\label{lem: theLSCLemma}
Suppose $\{g_k\}^{\infty}_{k = 1} \subset Z_{\text{ad}}$ is a sequence converging to $g$ uniformly in $\Om \times (-M, M)$ for all $M > 0$. Further suppose that $\{u_k\}^{\infty}_{k = 1} \subset \widetilde{X}(\R^n)$ is a sequence such that, for all $k \in \N^+$ and $v \in \widetilde{X}(\R^n)$, we have
    \begin{equation}\label{eq: EnergiesNLgIneq}
        (J^K_{g_k})'(u_k)v \ = \ 0;
    \end{equation}
    that is, $(u_k, g_k) \in \Alg$ for all $k \in \N^+$. Finally assume that $u_k \rightharpoonup u$ weakly in $\widetilde{X}(\R^n)$ to some $u \in \widetilde{X}(\R^n)$. Then $(u, g) \in \Alg$, that is for all $v \in \widetilde{X}(\R^n)$ we have
    \begin{equation}\label{eq: EnergiesNLgIneqDes}
        (J^K_g)'(u)v \ = \ 0.
    \end{equation}
\end{lemma}

\begin{proof}
Due to the convergence  $u_k \rightharpoonup u$ weakly in $\widetilde{X}(\R^n)$, we have that
    \begin{equation}\label{eq: theLSCLemmaEq2}
        \lim_{k \rightarrow \infty}\iint_{\R^{2n}}\frac{(u_k(x) - u_k(y))(v(x) - v(y))}{|x - y|^{n + 2s}}dxdy \ = \ \iint_{\R^{2n}}\frac{(u(x) - u(y))(v(x) - v(y))}{|x - y|^{n + 2s}}dxdy
    \end{equation}
    for all $v \in \widetilde{X}(\R^n)$. Then, owing to the formula \eqref{eq: First Frechet}, it suffices to prove that $g_k(\cdot, u_k(\cdot)) \rightarrow g(\cdot, u(\cdot))$ strongly in $L^1(\Om)$ as $k \rightarrow \infty$.
By the Triangle Inequality,
    \begin{equation}\label{eq: theLSCLemmaEq5}
        \|g_k(\cdot, u_k(\cdot)) - g(\cdot, u(\cdot))\|_{L^1(\Om)} \ \leq \ \|g_k(\cdot, u_k(\cdot)) - g_k(\cdot, u(\cdot))\|_{L^1(\Om)} + \|g_k(\cdot, u(\cdot)) - g(\cdot, u(\cdot))\|_{L^1(\Om)}.
    \end{equation}
    We claim that $\|g_k(\cdot, u(\cdot)) - g(\cdot, u(\cdot))\|_{L^1(\Om)} \rightarrow 0$ as $k \rightarrow \infty$. To see this, let $\ep > 0$ be fixed and for any $M > 0$ there exists a $N \in \N^+$ so that for all $k \geq N$ we have $|g_k(x, u(x)) - g(x, u(x))| \leq \frac{\ep}{|\Om|}$ for all $x \in \Om$ where $|u(x)| \leq M$. At the same time, due to \eqref{Eq: gqCoercive} and the boundedness of $\Om$ we have the growth bound
    \begin{equation}\label{eq: theLSCLemmaEq5A}
        |g_k(x, u(x)) - g(x, u(x)| \ \lesssim \ |u(x)|^{q - 1}
    \end{equation}
    for all $x \in \Om$, in particular when $|u(x)| > M$.\footnote{The underlying constant only depends on $|\Om|$.} Combining these observations gives us the estimate
    \begin{equation}\label{eq: theLSCLemmaEq5B}
        \|g_k(\cdot, u(\cdot)) - g(\cdot, u(\cdot))\|_{L^1(\Om)} \ \lesssim \ \ep + \int_{\Om \cap \{|u(x)|^{q - 1} > M^{q - 1}\}}|u(x)|^{q - 1}dx
    \end{equation}
    Since $u \in L^{2^*_s}(\Om)$ and $q < 2^*_s$, the integral on the right-hand side of \eqref{eq: theLSCLemmaEq5B} is finite, and moreover, converges to $0$ as $M \rightarrow \infty$ (and as $k \rightarrow \infty$, since $N$ depends on $M$). As a result,
    \begin{equation}\label{eq: theLSCLemmaEq5C}
        \limsup_{k \rightarrow \infty} \|g_k(\cdot, u(\cdot)) - g(\cdot, u(\cdot))\|_{L^1(\Om)} \ \leq \ \ep.
    \end{equation}
    Since $\ep > 0$ was arbitrary, we conclude that $\|g_k(\cdot, u(\cdot)) - g(\cdot, u(\cdot))\|_{L^1(\Om)} \rightarrow 0$ as $k \rightarrow \infty$.
    
Thus to complete the proof, we must show that $\|g_k(\cdot, u_k(\cdot)) - g_k(\cdot, u(\cdot))\|_{L^1(\Om)} \rightarrow 0$ as $k \rightarrow \infty$. Up to a not relabeled sub-sequence, $u_k \rightarrow u$ a.e. in $\Om$, so $g_k(x, u_k(x)) - g_k(x, u(x)) \rightarrow 0$ a.e. $x \in \Om$. To construct a  bound on the integrand $|g_k(x, u_k(x)) - g(x, u(x))|$ uniformly integrable in $k$, we use the growth assumption \eqref{Eq: gqCoercive} to see that
    \begin{equation}\label{eq: theLSCLemmaEq6}
        |g_k(x, u_k(x)) - g(x, u(x))| \ \leq \ a_1 + a_2|u_k(x)|^{q - 1} + a_2|u(x)|^{q - 1},
    \end{equation}
    and since $\Om$ is bounded and $q < 2^*_s$, we have that $a_1 + a_2|u_k(x)|^{q - 1} + a_2|u(x)|^{q - 1}$ is a suitable sequence of $L^1(\Om)$ integrands converging a.e. to $a_1 + 2a_2|u(x)|^{q - 1}$. Furthermore, due to the compact embedding $H^s(\Om) \Subset L^{2^*_s}(\Om)$, we have that $u_k \rightarrow u$ strongly in $L^{2^*_s}(\Om)$. Thus by the Generalized Dominated Convergence Theorem, we obtain $g_k(\cdot, u_k(\cdot)) \rightarrow g(\cdot, u(\cdot))$ strongly in $L^1(\Om)$, completing the proof.

\end{proof}

To conclude this section we present the analogue of Lemma \ref{lem: theLSCLemma} for Problem \ref{Pr: lCtrlProb}. The proof is identical.

\begin{lemma}\label{lem: theLSCLemmaLoc}
Suppose $\{g_k\}^{\infty}_{k = 1} \subset Z^{\loc}_{\text{ad}}$ is a sequence converging to $g$ uniformly in $\Om \times (-M, M)$ for all $M > 0$. Further suppose that $\{u_k\}^{\infty}_{k = 1} \subset \widetilde{H^1}(\R^n)$ is a sequence such that, for all $k \in \N^+$ and $v \in \widetilde{H^1}(\R^n)$, we have
    \begin{equation}\label{eq: EnergiesLgIneq}
        J^{\loc}_{g_k}(u_k) \ \leq \ J^{\loc}_{g_k}(v);
    \end{equation}
    that is, $(u_k, g_k) \in \Alg^{\loc}$ for all $k \in \N^+$. Finally assume that $u_k \rightharpoonup u$ weakly in $H^1(\Om)$ to some $u \in \widetilde{H^1}(\R^n)$. Then $(u, g) \in \Alg^{\loc}$, that is for all $v \in \widetilde{H^1}(\R^n)$ we have
    \begin{equation}\label{eq: EnergiesLgIneqDes}
        J^{\loc}_g(u) \ \leq \ J^{\loc}_g(v).
    \end{equation}
\end{lemma}


\section{Existence of Solutions to Control Problems}\label{Sec: Existence}
\subsection{Main Existence Results}\label{Subsec: Existence}

The following is the main result of this paper, existence of solutions for our fractional control problem.

\begin{theorem}[Existence of minimizers for nonlocal control problem]\label{Thm: ExMinNonlocalProb}
    Let $K: \R^n\setminus\{0\} \rightarrow \R$ be a kernel satisfying Assumption \ref{Assump: ker}. Then there exists a solution $(\overline{u_s}, \overline{g_s}) \in \Alg$ to Problem \ref{Pr: NlCtrlProb}. 
\end{theorem}

\begin{proof}
    Due to non-negativity, the cost functional \eqref{eq: CostFunc} is bounded from below over pairs in $\Alg$. Let $m := \inf_{(u, g) \in \Alg}I(u, g)$, and pick a minimizing sequence $\{(u_j, g_j)\}^{\infty}_{j = 1} \subset \Alg$ so that $\lim_{j \rightarrow \infty}I(u_j, g_j) = m$. Owing to the structure of the cost functional \eqref{eq: CostFunc}, the sequence $\{u_j\}^{\infty}_{j = 1} \subset \widetilde{X}(\R^n)$ is a bounded sequence in $\widetilde{X}(\R^n)$, so by reflexivity there exists a non-relabeled sub-sequence and a function $\overline{u_s} \in \widetilde{X}(\R^n)$ such that $u_j \rightharpoonup \overline{u_s}$ weakly in $\widetilde{X}(\R^n)$, and so we have
    \begin{equation}\label{Eq: ExistEq0}
        \liminf_{j \rightarrow \infty}[u_j - u_{\des}]^2_{X(\Om)} \ \geq \ [\overline{u_s} - u_{\des}]^2_{X(\Om)}.
    \end{equation}
    
    
    Next, due to $\{g_j\}^{\infty}_{j = 1} \subset Z_{\text{ad}}$ being a bounded sequence in $W^{1, r}(\Om \times (-M, M))$ for every $M > 0$ and Theorem \ref{P4}, there exists a further, not relabeled sub-sequence and a $\overline{g_s} \in Z_{\text{ad}}$ so that $g_j \rightarrow \overline{g_s}$ uniformly in $\Om \times (-M, M)$, for any $M \in \R$. Repeating the argument surrounding \eqref{eq: theLSCLemmaEq5C}, and using that $S \in L^q(\Om)$, we have that
    \begin{equation}\label{Eq: ExistEq0A}
        \lim_{j \rightarrow \infty}\int_{\Om}|g_j(x, S(x))|dx \ = \ \int_{\Om}|g(x, S(x))|dx.
    \end{equation}
    Combining \eqref{Eq: ExistEq0} and \eqref{Eq: ExistEq0A} yields
    \begin{equation}\label{Eq: ExistEq1}
        m \ = \ \lim_{j \rightarrow \infty}I_s(u_j, g_j) \ \geq \ I_s(\overline{u_s}, \overline{g_s}),
    \end{equation}

    Finally, due to the closedness of $Z_{\text{ad}}$ (Lemma \ref{lem: closedAndConvex}), $\overline{g_s} \in Z_{\text{ad}}$ as well. Then due to Lemma \ref{lem: theLSCLemma} we have that $(\overline{u_s}, \overline{g_s}) \in \Alg$, completing the proof.
\end{proof}

\begin{theorem}[Existence of minimizers for local control problem]\label{Thm: ExMinLocalProb}
    There exists a solution $(\overline{u}, \overline{g}) \in \Alg^{\loc}$ to Problem \ref{Pr: lCtrlProb}. 
\end{theorem}

 The proof technique is the same as for the nonlocal problem, but using Lemma \ref{lem: theLSCLemmaLoc} instead of Lemma \ref{lem: theLSCLemma}. Also, just as in the nonlocal setting, solutions to Problem \ref{Pr: lCtrlProb} need not be unique, because $\Alg^{\loc}$ is not convex.




\subsection{Discussion}\label{Subsec: ExistDiscuss}

In this subsection we discuss extensions and limitations of the theory used to prove Theorems \ref{Thm: ExMinNonlocalProb} and \ref{Thm: ExMinLocalProb}. First, since the case where $K(x) = |x|^{-(n + 2s)}$ is of special interest, we state Theorem \ref{Thm: ExMinNonlocalProb} explicitly in this special case.

\begin{corollary}[Fractional kernels]\label{Cor: FracKernelsMPOC}
There exists a minimizer $(\overline{u_s}, \overline{g_s})$ to the cost functional \eqref{eq: CostFunc} (replacing $X(\Om)$ with $H^s(\Om)$), where the minimizer is taken over all pairs $(u, g) \in \widetilde{H^s}(\R^n) \times Z_{\text{ad}}$ that solve the equation
\begin{equation}\label{Eq: FracLaplace}
\begin{cases}
    (-\lap)^s u(x) \ = \ g(x, u(x)), \ x \in \Om \\
    \ \ \ \ \ \ \ \ \ \ \ u(x) \ = \ 0 \ \ \ \  \ \ \ \ \ \ \ \ \ \ \ \ \ \ x \in \R^n \setminus \Om
    \end{cases}
\end{equation}
in the weak sense.
\end{corollary}

Next, we discuss an alternative way to interpret our cost functional.

\begin{remark}[Compliance]\label{Rmk: compliance}
    We may re-interpret Problem \ref{Pr: NlCtrlProb} as one where the cost we minimize is of a compliance type. The reason is that if $u_{\des} := 0$ in \eqref{eq: CostFunc}, then for any pair $(u, g) \in \Alg$, the cost functional can be rewritten as
    \begin{equation}\label{Eq: complianceFormCost}
        I_s(u, g) \ := \ \int_{\Om}g(x, u(x))u(x)dx + \la \int_{\Om}|g(x, S(x))|dx.
    \end{equation}
    However, if we do this we must assume that $\la > 0$. Otherwise, the minimum value of the cost functional will always be $0$, owing to the original form of the cost functional \eqref{eq: CostFunc} and the fact that for any $g \in Z_{\text{ad}}$, $(g, 0) \in \Alg$. Also notice that condition \ref{A4} precludes $0$ from being an admissible control. 

    The same remark holds for Problem \ref{Pr: lCtrlProb}.
\end{remark}

Finally, we provide a series of other remarks intended to highlight the limitations of our theory.

\begin{remark}\label{Rmk: Perturbations}
We may introduce a lower-order term to the energy \eqref{eq: NLEnergy} of the form $-\frac{\eta}{2}\int_{\Om}u(x)^2dx$ for some $\eta > 0$, and there still exist nontrivial minimizers to this perturbed energy; see \cite{servadei2013variational}. The results in this paper can be repeated in that setting since our conditions on the nonlinearity coincide with those utilized here.
\end{remark}

\begin{remark}\label{Rmk: noQuadII}
    If we use a cost functional of the form
    \begin{equation}\label{EqRmk: penaltyL2}
        I(u, g) \ := \ \frac{1}{2}\|u - u_{\des}\|^2_{L^2(\Om)} + \int_{\Om}|g(x, S(x))|dx,
    \end{equation}
    then we cannot guarantee that the minimizing sequence of states $\{u_j\}^{\infty}_{j = 1}$ generated in the proof of Theorem \ref{Thm: ExMinNonlocalProb} is bounded in the $X(\Om)$ norm. Unlike when the state equation is linear, we cannot get such a boundedness result a priori from Equation \ref{eq: NLWkStateEqn} because the degree of the nonlinearity on the right-hand side is to too high. Indeed, if $(u, g) \in \Alg$ is a nontrivial solution to \eqref{eq: NLWk} and we set $v := u$, then by \eqref{Eq: gqCoercive}, the Nonlocal Poincaré Inequality \cite[Theorem 6.7]{di2012hitchhikers}, and \ref{P5}, we obtain
    \begin{equation}
        [u]^2_{X(\Om)} \ \lesssim \ [u]^{2^*_s}_{X(\Om)},
    \end{equation}
    which is insufficient to get boundedness of states in $X(\Om)$.
\end{remark}

\section*{Acknowledgments}

JMS did not receive any external funding that was used to complete this project. Also, there was no data collected in the completion of this project.

\section*{Funding and/or Conflicts of interests/Competing interest}

JMS reports that there are no conflicts of interest or competing interests for this publication.


\end{document}